\documentclass[reqno,11pt,a4paper]{amsart}
\usepackage[a4paper,top=3cm,bottom=3.5cm,left=2.3cm,right=2.3cm]{geometry}
\usepackage{datetime}
\usepackage{color}

\usepackage{mathpazo}
 \usepackage{amsthm,amsmath,amssymb}
\usepackage{mathrsfs}
\usepackage{amsfonts}

\DeclareMathOperator{\espo}{e}
\newcommand{\eap}{\espo_{\textup{ap}}}
\newcommand{\expap}{\exp_{\textup{ap}}}
\usepackage{cite}

\newcommand{\p}{\partial}

\renewcommand{\H}{\mathcal{H}}

\renewcommand{\O}{\mathcal{O}}

\newcommand{\scalar}[2]{\langle#1,#2\rangle}

\newcommand{\step}[1]{\par\medskip\noindent\it#1\rm}
\newcommand{\Eucl}{\textup{Euc}}

\newcommand{\R}{\mathbb{R}}
\newcommand{\C}{\mathbb{C}}
\newcommand{\N}{\mathbb{N}}

\renewcommand{\t}{{\tau}}
\newcommand{\g}{\mathfrak{g}}

\newcommand{\e}{\varepsilon }

\newcommand{\I}{\mathcal{I}}
\renewcommand{\d}{\delta}

\DeclareMathOperator{\Lie}{Lie}
\renewcommand{\r}{{\rho}}

\renewcommand{\a}{\alpha}
\renewcommand{\b}{\beta}

\DeclareMathOperator{\Lip}{Lip}
\DeclareMathOperator{\Span}{span}

\newcommand{\abs}[1]{\lvert#1\rvert}

\newcommand{\norm}[1]{ \lVert#1\rVert}

\allowdisplaybreaks[1]
\newtheorem{theorem}{Theorem}[section]

\newtheorem{corollary}[theorem]{Corollary}

\theoremstyle{remark}

\theoremstyle{definition}

\theoremstyle{definition}
\newtheorem{example}[theorem]{Example}

\numberwithin{equation}{section}
\begin{document}
\title[A Hadamard-type  theorem for submersions and applications  to control theory]
 {A Hadamard-type open map Theorem for submersions\\
 and applications to completeness results in control theory}
 \author{Andrea Bonfiglioli, Annamaria Montanari and Daniele Morbidelli}
 \thanks{
  \textit{Acknowledgements.} The authors are members of the Gruppo Nazionale per
   l'Analisi Matematica, la Probabilit\`a e le loro Applicazioni (GNAMPA)
   of the Istituto Nazionale di Alta Matematica (INdAM).}
 \address{Dipartimento di Matematica,
         Universit\`{a} degli Studi di Bologna\\
         Piazza di Porta San Donato, 5 - 40126 Bologna, Italy.
         }
 \email{andrea.bonfiglioli6@unibo.it, annamaria.montanari@unibo.it, daniele.morbidelli@unibo.it
 }

\date{\today,\;h \currenttime}
\begin{abstract}
 We  prove a quantitative openness theorem for  $C^1$ submersions under suitable assumptions on the differential. We then apply our result to a class of exponential maps
 appearing in
Carnot-Carath\'eodory spaces and we improve  a classical
completeness result by Palais.
\end{abstract}

\maketitle

 \section{Introduction}\label{sec:introduction}
 It is well known by the Hadamard global inverse map theorem that
 {\it a $C^1$ local  diffeomorphism $f:\R^N\to \R^N$ with
  $\sup_{\R^N}\abs{df(x)^{-1}}<\infty$ is in fact a global diffeomorphism. }
Instead, if we consider a submersion $f:\R^N\to \R^n$ with
 $N> n$, one can not expect injectivity, but it is reasonable to hope that suitable
 conditions on
 the differential $df(x)$ may ensure that the map $f$ is onto.  This has been discussed in very general setting in the
 recent paper \cite{Rabier97}.

 The first purpose of  this paper is to prove  a   quantitative openness result for $C^1$ submersions, which in particular gives  a surjectivity theorem assuming uniform openness of the differential map. Our
 second task is to apply such result
to a class of exponential maps appearing in the a\-na\-ly\-sis of
Carnot-Carath\'eodory spaces. This will enable us to obtain an improvement of  a classical
completeness result for vector fields originally due to Palais \cite{Palais}.

 Here is our  first result.

 \begin{theorem}\label{susso}
  Let $M$ be a (finite dimensional) manifold of class $C^2$ and let $g$ be
  {a $C^1$ Riemannian}
  metric on $M$.
  Assume that we are given   positive
 constants
 $C_0<\infty $ and $r_0\leq \infty$, a point $x_0\in\R^q$ and a  map $f:B_\Eucl(x_0,r_0)\to M$  which is $C^1$ regular on the Euclidean
 ball $B_\Eucl(x_0, r_0)\subset\R^q$ such that
the differential (tangent map) of $f$ satisfies
\begin{equation}\label{infinitesimal}
 df(x)(B_\Eucl(0, C_0))\supseteq \{\xi\in T_{f(x)}M: g(\xi,\xi)\le 1\}\quad\text{for all $x\in B_\Eucl(x_0, r_0)$.}
\end{equation}
Then
\begin{equation}\label{quantitativeest}
 f(B_\Eucl(x_0, r_0))\supseteq B_{\mathrm{Rie}}\Bigl(f(x_0),\frac{r_0}{2C_0}\Bigr).
\end{equation}
\end{theorem}
In the statement $T_zM$ is the tangent space to $M$ at $z$,
 $B_{\mathrm{Rie}}$ denotes the ball in $M$ with respect to the Riemannian distance generated by $g$.


 Note that the theorem holds true with $r_0=\infty$ and takes the following form:   if $f:\R^q\to (M,g)$ is a $C^1$ submersion and it  satisfies
\eqref{infinitesimal}  for all $x\in\R^q$,  then the map $f $ is onto.
 We observe that, as it will appear in the proof of Theorem \ref{susso},
 our techniques also apply to the case when the source space
 is a Riemannian manifold, but we do not pursue this issue here.
On a concrete level, we will read \eqref{infinitesimal} as a uniform bound of the Moore--Penrose pseudoinverse of $df(x)$ (see the discussion in Section \ref{eleme}).

Topological properties of  a submersion $f$ between Riemannian manifolds
 have been discussed by Hermann \cite{Hermann60}, who
 exploited the notion of \emph{horizontal path-lifting}, requiring that~$f\in C^{1,1}$ and that~$df(x)\Bigr|_{\ker d
 f(x)^\perp}$ is an isometry.  Earle and Eells \cite{EarleEells}  generalized  Hermann's condition  to the infinite dimensional setting, showing that
 horizontal
 path-liftings can be constructed under the assumption  that~$df(x)$  admits a Lipschitz continuous family~$s(x)$ of right inverses. It must be said that
 the mentioned papers  are not primarily concerned with quantitative openness or surjectivity, but they are more focused on the
 topological/differential structure of the fibers of $f$.
 More recent results are due to  Rabier \cite{Rabier97}, who, in order
 to have surjectivity of a
 $C^{1,1}$ submersion $f$ from infinite dimensional Finsler manifolds,
 requires the existence of a Lipschitz-continuous right-inverse satisfying a suitable global upper bound.
This somewhat forces the assumption $f\in C^{1,1}$.  The only reference dealing with $C^1$ maps (between linear spaces
 only) is \cite{Rabier}. That paper does not rely on path liftings and  it does not seem to provide the
 quantitative  estimates \eqref{quantitativeest} of Theorem \ref{susso}

In this paper we shall work in finite dimension, but  without requiring the  $C^{1,1}$ regularity, because in the
applications simple examples show that our submersions  are at most $C^{1,\alpha}$  for some $\a<1$ (see the discussion before the statement of Theorem \ref{mainrough}).
In this case, horizontal  liftings are not necessarily unique, see Example \ref{nonuni} below.

Next we discuss how we  are able to apply Theorem \ref{susso} to Control Theory.
According to a remarkable result of Palais \cite{Palais}, given a family $\H=\{X_1,\dots,X_m\}$ of  \emph{complete} vector fields on a manifold $M$ generating  a
 \emph{finite dimensional } Lie algebra $\g$, then any
 vector field $Y\in\g$  is complete.  Recall that we say that
a vector field $Z$ on a manifold $M$ is complete if for each $x\in M$ the integral curve $t\mapsto e^{tZ}x$ is defined for $t\in\mathopen]-\infty,+\infty\mathclose[$.
 For the sake of brevity, we shall refer
  to this result as \emph{Palais' Completeness Theorem}.
%

 Our aim  is to provide an improvement of this theorem, in that we consider lower
 regularity
 assumptions
 on the vector fields and, most importantly, we replace finite-dimensionality with the following
 involutivity hypothesis: there is $s\in\N$ such that the family $ Y_1,\dots, Y_q $ of the nested commutators  (see \eqref{nested}) of length at most $s$ of the original vector fields of $\H$ satisfies
 \begin{equation}\label{involtino}
 [Y_i,Y_j](x)=\sum_{k=1}^q c_{i,j}^k(x)\,Y_k(x),\quad\text{ for all  $i,j=1,\ldots,q\quad x\in M$, }
  \end{equation}
for suitable  functions $c_{i,j}^k $  \emph{globally bounded} on $M$. In the sequel, we will say that $\H$ is $s$-involutive if \eqref{involtino} holds true.

Our approach is different from Palais' and from the approaches
 of the existing alternative proofs of Palais' Completeness Theorem.
 At our knowledge, all existing proofs of the Palais Theorem rely on ideas similar to each other, making a
 strong use  the Third Theorem of Lie;  see \cite[p.~145]{ChuKobayashi}, \cite[p.~966]{Jouan}, \cite[p.~96]{Palais}, \cite[p.~147]{Varadarajan}).
%
%
Our more general assumptions, which allow the coefficients $c_{i,j}^k$ to be non-constant, prevent us from applying the Third Theorem of Lie to the Lie algebra generated by $\{X_1,\ldots,X_m\}$, which may not be finite-dimensional
(see Example \ref{cierre}).

 Let $\H=\{X_1,\dots, X_m\}$ be a given family of
 complete vector fields on $M$ and assume  that the involutivity condition \eqref{involtino} holds
 for some $s\ge 1$. Note that requiring that the coefficients $c_{i,j}^k$ are constant on $M$ (a particular case of our assumptions) means that $\Lie\H$ is finite
 dimensional. Look at the non-autonomous Cauchy problem
 \[
  \dot\gamma=\sum_{k=1}^q b_k(t) Y_k(\gamma)\quad\text{for a.e. $t$}\quad\gamma(0)=x\in M,
 \]
where  $\norm{b}_{  L^\infty(\R)}\le 1$. We will show that there are $\e,\d>0$
and vector fields $Z_1,\dots, Z_\nu\in\H$ such that  for all $x\in M$,  $T\in[0,\e]$ one can find $t_1,\dots, t_\nu\in \R$ such that we can write
\begin{equation}\label{full} \begin{aligned}
\gamma(T)&=e^{t_\nu Z_\nu}\cdots e^{t_1 Z_1}x
\quad \text{with the estimate}\quad \sum_{j=1}^\nu \abs{t_j}\le \delta                                   .   \end{aligned}
\end{equation}
This statement will be made precise in the language of Carnot--Carath\'eodory geometry.
Since the choice of $\e$,   $\d$ and $Z_1,\dots,Z_\nu$  is independent of $x\in M$,
an iteration of \eqref{full}  will prove that $\gamma(T)$ is defined for all  $T\in\mathopen]-\infty,+\infty\mathclose[.$
This will yield the completeness of $\gamma$, thus providing the mentioned improvement of Palais' Completeness Theorem.
%

The sequence of exponentials appearing in \eqref{full} will be organized as a composition of \emph{ap\-pro\-xi\-ma\-te exponential maps}, which now  we  informally describe.
Let $\H=\{X_1,\dots, X_m\}$ be our given family of vector fields of class $C^s$ on a $C^{s+1}$ manifold $M$ of dimension $n$ and assume that $\H$ satisfies the $s$-involutivity hypothesis \eqref{involtino}.  Let  $\expap(Y_j)x$ be the approximate exponential of the nested  commutator  $Y_j$ in \eqref{nested}, see e.g. \cite{MMPotential}. To fix the ideas, if $Y_j= [X_k, X_\ell]$,  $t\geq 0$ and  $x\in M$, then
\[
 \expap(t Y_j)x:=\expap(t[X_k,X_\ell])x:=e^{-\sqrt t X_\ell}e^{-\sqrt t X_k}
 e^{\sqrt t X_\ell}e^{\sqrt t X_k} x.
\]
 (The  construction  of the   family $Y_1,\dots, Y_q$ is described in detail at the beginning of Section \ref{laterza}.) 
The definition can be generalized to $t<0$ and to commutators of higher length, using more complicated compositions of elementary exponentials. This will be precisely defined in Section~\ref{sottosezione}.

Introduce the map $E_x:\R^q\to M$,
\begin{equation}\label{esso}
 E_x(h):=\expap(h_1 Y_1)\circ\cdots\circ\expap(h_qY_q)x.
\end{equation}
Note that  the point $E_x(h)$ is actually of the form  appearing  in \eqref{full}, although the number of variables $q$ can be much larger than the dimension of $M$. Using a  first order expansion
for the tangent map of $E_x$, we will recognize that the map $E_x$ is a submersion from a neighborhood of the origin to  a suitable  submanifold of $M$ (the Sussmann orbit of the system $\H$ containing $x$). Moreover,
if  the target manifold is equipped with a suitable Riemannian metric,~$E$ satisfies the quantitative assumptions of
Theorem \ref{susso} . Concerning regularity,  in
 \cite[Example~5.7]{MM}) it is shown that we can expect at most that  $E\in C^{1,\alpha}$ for some $\a<1$.
 Thus we obtain the following:

%

\begin{theorem} \label{mainrough}
 Let $\H=\{X_1,\dots, X_m\}$ be a family of  complete vector fields of class  $C^s$ on a manifold~$M$ of class
 $C^{s+1}$ and assume that $\H$ is $s$-involutive (i.e., it satisfies \eqref{involtino}). Let $\rho$ be the Carnot--Carath\'eodory distance associated with
 the vector fields $ Y_1,\dots,Y_q $. If the functions $c_{i,j}^k $ in \eqref{involtino} are
 globally bounded on~$M$, then
 there are  absolute constants $\e,\delta> 0$ such that
\begin{equation}\label{expo}
  E_{x}( B_\Eucl(0,\e))\supseteq B_\r(x,\delta)\quad\text{for all $x\in M$.}
\end{equation}
\end{theorem}
In the language of Carnot--Carath\'eodory analysis, Theorem \ref{mainrough} can be referred to  as a ball-box theorem.
The new feature here is the uniformity of the constants $\varepsilon,\delta$ with respect to $x\in M$, which can be proved
as a consequence of the global boundedness of the coefficients in \eqref{involtino}.
From a technical point of view, the novelty is that we are using submersions
instead of local diffeomorphims (which are usually obtained by a restriction of $E_x$ to linear subspaces whose choice is quite delicate).
Ultimately, our arguments, although  they do not provide injectivity results,  seem to be simpler than   previous ones (compare to~\cite{NagelSteinWainger,MM}).

Our improved version of Palais' Completeness
Theorem follows as a corollary.
\begin{corollary}[Improved Palais' theorem]\label{pallina}
 Under the hypotheses of Theorem \ref{mainrough}, the Cauchy problem
 \begin{equation}\label{odessa}
  \dot\gamma=\sum_{j=1}^q b_j(t) Y_j(\gamma)\quad\text{a.e.}\quad\text{with}\quad \gamma(0)=x
 \end{equation} \label{cauchy}
has global solution for all  $x\in M$ and for each  $b\in L^\infty_{loc}(\R)$.
 \end{corollary}
A solution to the non-autonomous Cauchy problem \eqref{odessa} is  a continuous path $\gamma$
on $[0,T]$ which solves the integral equation
\[
 \gamma(t)=x+\int_0^t \sum_j b_j(s)Y_j(\gamma(s))ds\quad\text{for all $t\in[0,T]$.}
\]
Local existence and uniqueness follow from standard  theory (e.g., the  Picard iteration scheme).
It turns out that $\gamma$ is absolutely continuous and the ODE in \eqref{odessa} holds for a.e.~$t$. Moreover, in spite of the irregularity of the coefficients $b_j$, since the vector fields $Y_j$ are $C^1$, the classical results on dependence on initial data hold and it turns out that  the solution depends in a $C^1$ way from the initial data $x$. This can be seen by carefully checking that the   proof for autonomous systems (see \cite[Chapter~17]{Hirsch}) also
works for a non-autonomous system like \eqref{odessa}.

Before closing this introduction, we mention some applications of Palais' Completeness Theo\-rem (these applications
 also give a motivation for our investigation of this remarkable result), which is often applied jointly with
 \emph{Palais' Integrability Theorem} (as referred to e.g., in \cite{AbouqatebNeeb}): the latter theorem states that,
 given a family of  complete vector fields generating a finite-dimensional
 Lie algebra $\g$, $\g$ can be integrated to a global action of a Lie group.
 Palais' Completeness/Integrability Theorems
  naturally intervene in
   the study of Lie transformation groups of geometric structures on a manifold, see
   \cite[Theorem H]   {ChuKobayashi}. They also find
 many applications in Control Theory:
   \begin{itemize}
     \item[-] see the controllability results in \cite[Section 3]{Hirschorn}, furnishing explicit
     characterizations for the reachable sets;

     \item[-] see the feedback-stabilization results for ODE systems in \cite[Theorem 3.1]
     {MichalskaTorres};

     \item[-] see \cite[Theorems 5.1 and 6.1]{Jouan}, furnishing a characterization of the class
     of
     control systems
     which are globally diffeomorphic to a linear system (on a Lie group or a homogeneous space)
     in terms of finite-dimensional Lie algebras of complete vector fields (this is precisely the
     assumption in
     Palais' Completeness Theorem).
   \end{itemize}

  Furthermore, let us also mention that Palais' Theorem finds applications in the study of ordinary/stochastic differential equations
  (see e.g., \cite{LazaroOrtega, MannoOliveriVitolo}).
 We finally remark that Palais' Completeness and Integrability Theorems have been recently
 generalized
 also to the framework of infinite-dimensional algebras/manifolds, see \cite{AbouqatebNeeb, Leslie}.
 \medskip

%
%

 \section{Open map Theorem for submersions}

\subsection{Linear algebra preliminaries. } \quad \label{eleme}

\step{Fact 1.} Let $V,W$ be finite dimensional linear spaces and let
 $g$ be an inner product on $V$. If $T:V\to W$ is a linear map and  $T(V)=W$, then there is a
 unique
 \emph{Moore--Penrose  pseudoinverse} map  $T^{\dag}:W\to V$ such that  for each $w\in W$,
 $T^\dag(w)=v_{\textup{LN}}$ is the minimal norm solution 
  (here `LN' stands for `least norm')
 of the system $T(v)=w$. The solution
 $v_{\textup{LN}}$ is characterized by the conditions
 \[
  T(v_{\textup{LN}})=w\quad\text{and}\quad v_{\textup{LN}}\perp\ker T.
 \]
In particular, $T^\dag$ is a right inverse of $T$, i.e. $TT^\dag=I_W$.

\step{Fact 2.} If $p\le q$ and $A\in \R^{p\times q}$ has independent rows, then the Moore--Penrose
 pseudoinverse $A^\dag$ is a right inverse of $A$ and
 has the explicit form
 \[
  A^\dag =A^T(AA^T)^{-1}.
 \]
For each $y\in\R^p$ the minimal norm solution $x_{\textup{LN}}$ of the system $A x=y$ has the form $x=A^\dag y$ and, among all solutions of such system, it  is characterized by the orthogonality condition
$x_{\textup{LN}}\perp \ker A$.

\step{Fact 3.} If $v_1,\dots, v_q$ span a $p$-dimensional linear  space $W$, then we can define a natural  inner product $g$ on $W$ by letting for all $v,v'\in W$
\begin{equation}\label{gigo}
 g(v,v'):=\scalar{T^\dag v}{T^\dag v'}_{\R^q},
 \end{equation}
where $T^\dag$ is the pseudoinverse of the linear map $T:\R^q\to W$, defined by $T(\xi ) =\sum_{j=1}^q v_j\xi_j$.
If $p=q$, then $T^\dag=T^{-1}$ and the vectors $v_1,\dots, v_p$ turn out to be orthonormal. Note that $g$ is nondegenerate because $TT^\dag=I_W$.

\step{Fact 4.}  If $A$ is  a matrix with independent rows, we have $
\abs{ A^\dag}^{-1}=\min_{\abs{y}=1} \abs{A^Ty}$.
Indeed,  let $n\le N$ and let $A\in\R^{n\times N}$. First, it is well known that
$A$  is onto if and only if   $A^T$ is one-to-one (just because~$\operatorname{Im} A=(\ker A^T)^\perp$).
The required equality can be checked as follows.
First observe that, since $A$ has full rank and $n\le N$,  $AA^T$ is positive-definite, and moreover,
$A^\dag=A^T(AA^T)^{-1}$.  Thus,
\[
\begin{aligned}
\abs{A^\dag}^2 & =\max_{\abs{y}=1}\abs{A^T(AA^T)^{-1}y}^2=\max_{\abs{y}=1}\scalar{A^T(AA^T)^{-1}y}{A^T(AA^T)^{-1}y}
  \\&
 =\max_{\abs{y}=1}\scalar{ (AA^T)^{-1}y}{ y}=\lambda_{\textup{max}}
 ((AA^T)^{-1})=\frac{1}{\lambda_\textup{min}(AA^T) },
\end{aligned}
\]
where $\lambda_{\textup{max}}$ and $\lambda_{\textup{min}}$ denote the largest/smallest eigenvalue. On the other hand,
\[
\begin{aligned}
 \min_{\abs{y}=1}\abs{A^Ty}^2 = \min_{\abs{y}=1}\scalar{AA^Ty}{y}=\lambda_{\textup{min}}
 (AA^T).
\end{aligned}
\]
Thus \emph{Fact 4} is checked.

 \subsection{Proof of Theorem \ref{susso}}
  \label{openn}
  Here we prove Theorem \ref{susso}. We will use a horizontal path-lifting argument choosing
  the Moore--Penrose pseudoinverse as a right inverse of the differential. Indeed, the requirement \eqref{infinitesimal} is equivalent to
 the assumption that   $df(x):T_x\R^q\to T_{f(x)}M$ is onto for all $x\in
  B_\Eucl(x_0,r_0)$ and  that the estimate\footnote{An equivalent assumption is   $\inf_{\abs{\xi}=1}\abs{df(x)^T\xi}\ge C_0^{-1}$ for all $x$. See the linear algebra Fact 4 in Section \ref{eleme}.}
  $\abs{df(x)^\dag}\le C_0$ holds true for all $x\in B_\Eucl(x_0,r_0)$, where
  $\abs{df(x)^\dag}$ denotes the operator norm of the Moore--Penrose pseudoinverse of $df(x)$.


 \begin{proof}[Proof of Theorem \ref{susso}]
Assume that $x_0=0$ and denote  briefly $B_0:=B_\Eucl(0, r_0)\subset\R^q$. Let $y\in B_{\mathrm{Rie}}(f(0), \frac{r_0}{2C_0})$. Thus there is  $r_1< \frac{r_0}{2C_0}$ and  an absolutely continuous path  $\gamma:[0,1]\to M$   such that
 \begin{equation*}
  \gamma(0)=f(0),\; \abs{\dot\gamma(t)}_g\le r_1 \quad\text{for a.e. $t$, and}\; \gamma(1)=y.
 \end{equation*}
  Our strategy relies on the construction of a lifting of $\gamma$. Since there is no uniqueness, it
  is
  crucial to make a ``minimal''  choice (think to the example   $f:\R^2\to\R$ with $f(x_1, x_2) =
 x_1+\arctan x_2 $).
Precisely,
 we claim that there is an absolute continuous path $ \theta:[0,1]\to B_0$
such that
\begin{subequations}\begin{align}
 &\theta(0)=0,\label{first}
 \\&
    f(\theta(t))=\gamma(t)\;\text{for all $t\in[0,1]$,}\label{second}
\\&
 \dot\theta(t)\perp\ker  df(\theta(t)) \quad\text{for a.e. $t\in[0,1]$.}\label{third}
        \end{align}
        \end{subequations}
We refer to the last line as a \emph{minimality condition}. We remark that in the literature such condition is  also called \emph{horizontality condition}.
Note that if $[a,b]\subset[0,1]$ and  an absolutely continuous path $\theta: [a,b]\to B_0$ satisfies \eqref{second} and \eqref{third}, then, differentiating, we get
\begin{equation}\label{duff}
 \abs{\dot\theta(t)}=\abs{df(\theta(t))^\dag\dot\gamma(t)}\le C_0 r_1
 \quad\text{for a.e. $t\in[a,b]$.}
 \end{equation}
 Thus, $\theta$ is Lipschitz continuous with Lipschitz constant $C_0r_1$.

Let \begin{equation}\begin{aligned}
\overline b:=\sup &\Bigl\{b\in[0,1]\text{ such that exists $\theta\in \Lip ([0,b], B_0 )$  with $\theta(0)= 0 $}
\\&\qquad f\circ\theta=\gamma \text{ on $[0,b]$ and $\dot\theta\perp \ker df(\theta)$ a.e. on $[0,b]$} \Bigr\}  .    \end{aligned}
\end{equation}
Our aim is to prove that $\overline b=1$.

\step{Step 1.} We first show that $\overline b>0.$

Fix coordinates $\xi:U\to V$ from an open set $U\subset M$ containing $f(0)$ to an open set $V\subset\R^p.$ Let $\Omega\subset B_0 $ be the connected component of $f^{-1}(U)$ containing the origin. Let  $\psi:=\xi\circ f:\Omega\to V$ and $\sigma(t)=(\xi\circ\gamma)(t)$ for all $t\in[0,1]$ such that  $\gamma(t)\in U$.  Denote   by
$J_\psi(\theta)\in \R^{p\times q}$  the Jacobian matrix of $\psi$ at the point $\theta\in\Omega$ and
look at the Cauchy problem
\begin{equation}\label{auch}
 \dot\theta(t)=J_\psi(\theta(t))^\dag \dot\sigma(t)\quad \text{a.e., with}\quad \theta(0)=0.
\end{equation}
Here  $J_\psi(\theta)^\dag\in\R^{q\times p}$ denotes
 the Moore--Penrose inverse   of  $J_\psi(\theta)$. Since  $J_\psi(\theta)$ has full rank $p\le q$, we
 have the explicit formula
\begin{equation}\label{jkss}
 J_\psi(\theta)^\dag=J_\psi(\theta)^T(J_\psi(\theta)J_\psi(\theta)^T)^{-1}\quad\text{for all $\theta\in\Omega$,}                                                                                                           \end{equation}
which shows that  the function $\theta\mapsto  J_\psi(\theta)^\dag$ is continuous. Therefore, the Cauchy problem \eqref{auch} has at least a solution $\theta:[0,b\mathclose[\to \Omega$    on some small nonempty interval $[0,b\mathclose[$.
More precisely, $\theta$ is absolutely continuous and  solves the integral equation
\[
 \theta(t)=\int_0^t J_\psi(\theta(\tau))^\dag\dot\sigma(\tau)d\tau \quad\text{for all $t\in[0,b\mathclose[$.}
\]
At any differentiability point $t$ we have
 \[
 J_\psi(\theta(t))\dot\theta(t)=J_\psi(\theta(t))J_\psi(\theta(t))^\dag\dot\sigma(t)= \dot\sigma(t),
 \]
which implies that $\psi\circ\theta=\sigma$ on $[0,b\mathclose[$.
Applying $\xi^{-1}$ we get $f\circ\theta=\gamma$.

Finally, to accomplish \emph{Step 1},  we check the orthogonality condition. Since $\theta$ solves the  problem  \eqref{auch}, we have for almost all $t$ the condition
$
  \dot\theta(t)\perp \ker J_\psi(\theta(t))$.
By the chain rule  $d\psi(\theta) =d\xi(f(\theta))df(\theta)$ and since $d\xi(y)$  is an invertible linear map for all $y\in U  $, we see that  $\ker J_\psi(\theta(t))= \ker df(\theta(t))$, as required.

\step{Step 2.} We show that $\overline{b}=1$.

Assume by contradiction that $\overline{b}<1$. Let $b_n\uparrow \overline{b}$ and let $\theta_n:[0, b_n]\to B_0 $ be a Lipschitz-continuous path such that $\theta_n(0)=0$,
 \[
  df(\theta_n)\dot\theta_n=\dot\gamma \text{ and }\dot \theta_n\perp\ker df(\theta_n) \text{ a.e. on  $[0, b_n]$.}
\]
 We have in fact $\dot \theta_n=(df(\theta_n))^\dag \dot\gamma(t)$
a.e.~on $[0, b_n]$.
Therefore, by \eqref{duff} we get the estimate
\[
 \abs{\theta_n(t)}\le C_0 r_1 \abs{t}\le\frac{r_0}{2},
\]
because we know that $r_1<  \frac{r_0}{2C_0} $. Therefore the image of $\theta_n$ is strictly inside the ball $B_0$.
Next, extend each $\theta_n$ on the whole interval $[0, \overline{b}]$ by letting  $\theta_n(t)=\theta_n(b_n)$ for $t\in[b_n, \overline{b}]$. The family $\theta_n$ is  uniformly bounded and  Lipschitz continuous (see \eqref{duff}). Thus,  by Ascoli--Arzel\`a theorem, up to extracting a subsequence, we may assume that $\theta_n$ converges uniformly on $[0,\overline b]$ to a limit function~$\theta.$
 Start  from the integral formula
\[
 \theta_n(t)=\int_0^t df(\theta_n(\tau))^\dag\dot\gamma(\tau)d\tau \quad\forall t\in[0, b_n],
\]
  and note that coordinate manifestations like \eqref{jkss} show that $\theta\mapsto df(\theta)^\dag$ is continuous.
Passing to the limit as $n\to +\infty$   (by dominated convergence), we discover that $\theta$ satisfies \eqref{first}, \eqref{second} and \eqref{third}  on $[0,\overline{b}]$ and has global Lipschitz constant $C_0 r_1 $  (again by \eqref{duff}).
Therefore the limit $\lim_{t\to \overline b-}\theta(t)=:\theta(\overline{b})$ exists and belongs to~$B_0$.

To conclude Step 2, fix coordinates $\xi':U'\to V'\subset \R^p$ on a neighborhood $U'$ of $f(\theta(\overline{b}))$ and, arguing as in \emph{Step 1}, construct for some $\e>0$
a Lipschitz solution $\Theta:[\overline b,\overline b+\e\mathclose[\to B_0$ such that
\[
 \Theta(\overline b)=   \theta(\overline{b}),\quad  df(\Theta(t))\dot\Theta(t)=\dot\gamma(t)
 \quad \text{and}\quad\dot\Theta(t)\perp\ker df(\Theta(t))\quad\text{for a.e. $t\in[\overline b,\overline b+\e]$.}
\]
Thus we have extended the solution $\theta$ on an interval strictly larger than $[0, \overline b]$ and we get a contradiction. This concludes the proof of \emph{Step 2}
 and, jointly, the proof of the theorem.
\end{proof}

\begin{example}\label{nonuni}
  We show that the horizontal lifting of a  path via a $C^1$ lifting is not unique. Consider the
 $C^1$ submersion  $f:\R^3\to \R$
  \[
  f(x_1,x_2,x_3)=        \frac 23x_1\abs{x_1}^{1/2} +   \frac 23x_2\abs{x_2}^{1/2}+x_3,
             \]
whose gradient is $\nabla f(x)=(\abs{x_1}^{1/2}, \abs{x_2}^{1/2},1)$.
Both the paths
\[
 \theta(t)=\Bigl(\frac 14 t^2,0,t\Bigr)\quad\text{and}\quad
  \phi(t)=\Bigl(0,\frac 14 t^2,t\Bigr)
\]
are liftings of the same path $f\circ\theta=f\circ\phi$. Moreover, they are minimal, because they are both integral curves of the continuous vector field $\nabla f$.
\end{example}

\begin{example}
 The ``horizontal bundle''  $(\ker df)^\perp$ is in general not involutive. For example, taking
 $
  f(x,y,z)=(x+yz,y)\in\R^2,$
one sees that  $(\ker df)^\perp$ is generated by
the vector fields $\p_y$
and $\p_x+z\p_y+y\p_z$, which form a non-involutive system.
\end{example}

\section{Applications to completeness results in control theory}

\subsection{$s$-involutive families of vector fields }
\label{laterza}
Let $ X_1,\dots,X_m $ be a family of vector fields on an $n$-dimensional
manifold $M$  and let $s\in\N$. Assume that $M$ is $C^{s+1}$ and that $X_j\in C^s(M)$ for $j=1,\dots,m$.
Fix a bijection between the  finite set of words  $\{w=w_1w_2\cdots w_\ell: w_j\in\{1,\dots,q\},\; \ell\in\{1,\dots, m\} \}$  and the set $   \{ 1,\dots, q \}$ where $q\in\N$ is suitable.  Enumerate accordingly  as $ Y_1,\dots, Y_q  $ all the  nested commutators up to length $s$ of the original vector fields $X_1,\dots, X_m$.
  Precisely, for each $j\in\{1,\dots, q\}$, $Y_j$ is a bracket of the type
\begin{equation}
\label{nested}
Y_j=[X_{w_1},[X_{w_2}, \dots, [X_{w_{\ell-1}}, X_{w_\ell}]\dots ]],
\end{equation}
where $w_1,\dots, w_\ell\in\{1,\dots,m\}$. In such case we say that $\ell$ is the length of $Y_j$.
Note that each $Y_j$ is a $C^1$ vector field.

  As in \cite{MMPotential}, we
  say that $\H $ is $s$-involutive if   $X_j\in C^s(M)$ for all $j$ and there are locally bounded functions
 $c_{i,j}^k$ (with $1\le i,j,k\le q$) so that
 \begin{equation}\label{integrabile}
  [Y_i,Y_j] (x)= \sum_{k=1 }^q c_{i,j}^k(x)Y_k(x) \quad\text{for all $i,j\in\{1,\dots,q\}.$ }
 \end{equation}
 Here,    in local coordinates we have $Y_j=\sum_{\a=1}^n
 Y_j^\a\p_\a$  and $Y_j^\a\in C^1$.
 Furthermore we set $[Y_i,Y_j]=\sum_\a (Y_iY_j^\a-Y_jY_i^\a)\p_\a$, which is a continuous vector field.
 Observe also that if
  $\dim(\Lie\H)<\infty$, then the family $\H $ is involutive for some $s$   and
  the functions $c_{i,j}^k$ can be chosen to be constant.
Moreover, the local boundedness of $c_{i,j}^k$ in~\eqref{integrabile} ensures that  the distribution generated by
the vector fields $ Y_j $ is integrable, by Hermann's Theorem \cite{Hermann}.
Thus $M$ can be decomposed as a disjoint union of orbits. Namely, we can define the orbit containing $x\in M$ as
\[
 \O_x:=\{e^{t_1 Z_1}\cdots e^{t_\nu Z_\nu}x :\nu\in\N, Z_1,\dots, Z_\nu\in\H\text{ and }(t_1,
 \dots, t_\nu)\in\Omega_{Z_1,\dots,Z_\nu,x}\},
\]
where $\Omega_{Z_1,\dots,Z_\nu,x}\subset\R^\nu$ is a suitable (maximal) open set containing the origin in $\R^\nu$ such that all the exponential maps are well defined.
Each orbit $\O$, equipped with a suitable topology (see \cite{MMPotential})
is  an immersed submanifold  of class $C^2$ of $M$ and it is  an integral manifold of the distribution generated by  the vector fields $Y_j$. In other words,
 $T_y\O= \Span\{Y_j(y):1\leq j\le q\}$ for all $y\in\O$.
  The number $ p_x:=\dim (\Span\{Y_j(x):1\leq j\le s\})$
can be different at different points but it is constant on a fixed orbit.


By the Linear Algebra preliminaries in the previous section, for any $y\in\O$ we can define an inner product $g$ on $T_y\O$ via the matrix $[Y_1(y),\dots, Y_q(y)]$ as in \eqref{gigo}.
It can be checked that this gives a $C^1$ Riemannian
metric on the $C^2$ manifold $\O$.
The corresponding norm of   $Z\in T_y\O$ is
\begin{equation}\label{dusss}
 \abs{Z}_y:=g(Z,Z)^{1/2}=\min\{\abs{\xi}:\sum_{j=1} Y_j(y)\xi_j=Z\}=\Bigl|[Y_1(y),\dots, Y_q(y)]^\dag Z\Bigr|
 _\Eucl.
\end{equation}
Note that   $\abs{Z}_y\le 1$ if and only if  $Z$ is subunit in the Fefferman--Phong sense with respect to the family
$Y_1,\dots, Y_q$ (see \cite{FeffermanPhong81}).
We denote by $\rho(x,y)$ the distance associated with such a Riemannian metric, which is nothing but the  Carnot--Carath\'eodory distance associated with the family of vector fields $Y_1,\dots, Y_q$.
 Precisely,
$
 \rho(x,y)$ is the infimum of all $T\ge 0$ such that one can find an absolutely continuous solution $\gamma:[0,T]\to M$  of
 the   problem  $\dot\gamma=\sum_jb_j(t)Y_j(\gamma)$ a.e., with $\gamma(0)=x$, $\gamma(T)=y$ and $\sum_j b_j^2(t)\le 1$ a.e.
 Clearly, $\rho(x,y) <\infty$  for each pair of points on the same orbit.

We introduce the following notation for multi-indexes.
Let   $p,q\in\N$ with  $ 1\le p\le q $. Define
\[ \I(p,q) :=\{(i_1,
\dots,i_p)  : 1\le i_1<i_2\cdots<i_p\le q)\}  . \]
Moreover, if $x\in M$, $p=p_x$ and  $I\in\I(p,q)$ we let
\[\begin{aligned}
Y_I(x):=Y_{i_1}(x)\wedge\cdots\wedge Y_{i_p}(x)\in
 \bigwedge_p T_x M.
  \end{aligned}
  \]
%
Note that $Y_I(x)\neq 0$ means that  $Y_{i_1}(x),\dots, Y_{i_p}(x)$ are independent. Therefore,  they generate~$T_x\O_x$, because $p=p_x$ is its dimension. Moreover, for each $K=(k_1,\dots, k_p)\in\I(p,q)$ we have $Y_{k_1}(x),\dots, Y_{k_p}(x)\in T_x\O$ and thus $ Y_{k_1} \wedge\cdots\wedge Y_{k_p}$ is a scalar multiple of $Y_{i_1}\wedge\cdots\wedge Y_{i_p} $ at $x$. Since the set $\I(p,q)$ is finite, at any point $x\in M$, there is $I\in\I(p,q)$ which is ``maximal" in the following sense:
\begin{equation}\label{massimetto}
\text{for all $K\in\I(p,q)$, we can write $Y_K(x)= c_{K,I}Y_I(x)$ with $|c_{K,I}|\le 1$.                                                                              } \end{equation}
The meaning of the maximality condition \eqref{massimetto} is  the following.
 Let  $x\in M$, $p=p_x$ and  $I\in\I(p,q)$ which  satisfies \eqref{massimetto} at $x$. Choose any independent $p$-tuple 
 $Y_{k_1}(x), \dots, Y_{k_p}(x)$. Then $Y_K(x)\neq 0$ and \eqref{massimetto} implies that $Y_I(x)\neq 0$ too. Therefore, 
 $Y_{i_1}(x),\dots, Y_{i_p}(x)$ are a basis of  $ \operatorname{span}\{Y_i(x):1\leq i\leq q\}$. Therefore,
for all $k\in\{1,\dots, q\}$, the solution $(b_k^1, \dots, b_k^p)\in\R^p$ of the linear system
$
 \sum_{\a=1}^p Y_{i_\a}(x)b_k^\a= Y_k(x)
$
is unique and
satisfies the estimate
\begin{equation}\label{massimino}
 \abs{b_k^\a}\le 1 \quad\text{for all $\a=1,\dots,p$.}
\end{equation} 
To see estimate \eqref{massimino}, just note that, at the point $x$,
given $k\in\{1,\dots,q\}$ and $\a\in\{1,\dots,p\}$,
\[\begin{aligned}
 Y_{i_1}\wedge\cdots\wedge Y_{i_{\a-1}}\wedge Y_k\wedge
 Y_{i_{\a+1}}\wedge\cdots\wedge Y_{i_p}
& =Y_{i_1}\wedge\cdots\wedge Y_{i_{\a-1}}\wedge \sum_\b b_k^\beta Y_{i_\beta}\wedge
 Y_{i_{\a+1}}\wedge\cdots\wedge Y_{i_p}
 \\&=b_k^\a Y_{i_1}\wedge\cdots\wedge Y_{i_p},
\end{aligned}
\]
and the  estimate   follows, if $I$ satisfies \eqref{massimetto} at $x$.

%
%
%

 \subsection{Proof of Theorem \ref{mainrough}}
 \label{sottosezione}

 Now we are in a position to get the proof of   Theorem~\ref{mainrough}. Our arguments
 rely on some first-order expansion  of a class of \emph{approximate exponential maps} which  was already discussed in
 \cite{MMPotential}.
  We describe here the main
  steps with some comments on the
uniform aspects  of the estimates displaying  in our setting, due to our global boundedness
  assumption $c_{i,j}^k\in L^\infty (M)$.
 Note that such uniformity can not  be expected if $c_{i,j}^k$ are locally bounded only.

 Let
\begin{equation}\label{global}
 C_1:=\sup_{x\in M} \abs {c_{i,j}^k(x)} <\infty \quad\text{for all $i,j,k=1,\dots,q$.}
\end{equation}

 \step{Step 1.}
 In order to denote appropriately the commutators, we use the following  notation:
 given a word $w=w_1w_2\cdots w_\ell$ of length $\ell$ in the alphabet $\{1,2,\dots,m\}$, we
 denote $$X_w=[X_{w_1},[X_{w_2},\dots,[W_{w_{\ell-1}},X_{w_\ell}]\dots]].$$
 Let $w_1,\dots, w_\ell\in\{1,\dots,m\}$ and define  for $\t\in\R$
 the following $C^1$  diffeomorphisms from $M$ to itself (see also \cite{LM}):
 \begin{equation*}
 \begin{aligned}
 C_\t( X_{w_1})& := \exp(\t X_{w_1}),
 \\ C_\t( X_{w_1}, X_{w_2})& :=\exp(-\t X_{w_2})\exp(-\t X_{w_1})\exp(\t
X_{w_2})\exp(\t X_{w_1}),
 \\&\vdots
  \\C_\t( X_{w_1}, \dots, X_{w_\ell})&
:=C_\t( X_{w_2}, \dots, X_{w_\ell})^{-1}\exp(-\t X_{w_1}) C_\t( X_{w_2}, \dots,
X_{w_\ell})\exp(\t X_{w_1}). \end{aligned}
 \end{equation*}
Then, given the word $w=w_1w_2\cdots w_\ell$, define the approximate exponential of the commutator    $X_w:=[X_{w_1},[X_{w_2},\cdots,
 [W_{w_{\ell-1}},X_{w_\ell}]]]$ by
\begin{equation}\label{appsto}
\eap^{tX_{w_1 w_2\dots w_\ell}} :=  \expap(t X_{{w_1 w_2\dots w_\ell}}):=
\left\{\begin{aligned}
& C_{t^{1/\ell}}(X_{w_1}, \dots, X_{w_\ell} ), \quad &\text{ if $t\geq 0$,}
\\
&C_{|t|^{1/\ell}}(X_{w_1}, \dots, X_{w_\ell} )^{-1}, \quad &\text{ if $t<0$.}
                 \end{aligned}\right.
\end{equation}
Note that, since our vector fields are complete by assumption,  $\expap(tX_{w_1 w_2\dots w_\ell}) $ is defined for all $t\in\R$.

Next, fix a word  $w=w_1\cdots w_\ell$  and denote   briefly $C_t:=C_t( X_{w_1}, \dots, X_{w_\ell})$. By \cite[Theorem~3.5]{MM} and \cite[Theorem~3.8]{MMPotential}, we have the expansions
\begin{equation*}
  \frac{d}{dt}f (C_t x )= \ell t^{\ell-1} X_w f(C_tx) +\sum_{|v|
  =\ell+1}^s
a_v t^{|v|-1}
  X_vf(C_tx)
+t^s\sum_{\abs{u}=1}^s b_u(x,t)X_uf(C_t x),
\end{equation*}
and
\begin{equation*} 
\begin{aligned}
  \frac{d}{dt}f(C_t^{-1} x) = & - \ell t^{\ell-1} X_w f(C_t^{-1}x)
+\sum_{|v|=\ell+1}^s   \widetilde a_v
  t^{|v|-1}
  X_v f(C_t^{-1}x)
  + t^s \sum_{\abs{u}=1}^s   \widetilde{b_u}(x,t)X_uf (C_t^{-1}x),
\end{aligned}
\end{equation*}
where $f:M\to \R$ is any $C^2$ test function. The sums on $v$ are empty if $\abs{w} = s $. If not, we have the cancellation property
$\sum_{\abs{v}=\ell+1}(a_v+  \widetilde a_v)X_v(x)=0$ for all $x\in M$.
In our context, the estimates on   $b_u$ and $  \widetilde{ b_u}$  are global as $x\in M$:
\begin{equation}
 \abs{b_u(x,t)}+
 \abs{\widetilde{ b_u}(x,t)}\le C\quad\text{for all $x\in M$ and
 $\abs{t}\le 1$.}
\end{equation}
Moreover, if the coefficients $c_{jk}^\ell$ are constant, $b_u$ and $\widetilde{b_u} $ are independent of $x$.  This can be seen by tracking the details of the proof in \cite[Theorem~3.8]{MMPotential}.
Note  that our arguments can not provide any estimate on    $b_u(x,t)$
for large $t$.

\step{Step 2.} Next we look at the adaptation to our setting of  Theorem~3.11 in \cite{MMPotential}.
Under the hypotheses of Theorem \ref{mainrough}, if $x\in M$, then the map
  $E_x $ introduced in \eqref{esso}
is globally defined (it is constructed by integral curves of the vector fields in $\H$ which are complete). Moreover  $E_x$
is $ C^1 $  smooth (see \cite[Theorem~3.11]{MMPotential}). Actually $E_x$ sends points of the
orbit $\O_x$ to points of $\O_x$. Finally, in terms of the homogeneous norm $\norm{h}:=
\max_j\abs{h_j}^{1/\ell_j}$, we have  the first order expansion for the differential (tangent map) $E_*$ of  $E:=E_x$
 \begin{equation}\label{quattordici}
E_{*}(\p_{h_k})
 =Y_{k}(E(h))+\sum_{\ell_j=\ell_{k}+1}^s a_k^j(h)Y_j(E(h))+
 \sum_{j=1}^q\omega_k^j(x,h)Y_j(E(h)),\quad\text{for all $h\in\R^q$,}
\end{equation}
where, for all  $j,k,i$, we have the estimates
\begin{equation}\label{stimette} \begin{aligned}
 \abs{a_k^j(h)}&\le C\norm{h}^{\ell_j-\ell_{ k}}\quad \text{for all $h\in
 \R^q$.}
\\
 \abs{\omega_k^j(x,h)}
 &\le C\norm{h}^{s+1-\ell_{k}}\quad\text{for all
$  x\in M$ and $h\in B_\Eucl(0,1)\subset \R^q$.}
  \end{aligned}
 \end{equation}
   Here,
 the  power-type functions $a_k^j$ are always independent of $x\in M$.
 An examination of the arguments in \cite{MMPotential} shows that, under
 the hypotheses of Theorem \ref{mainrough} ($X_j$ global and $c_{j,k}^\ell \in L^\infty(M)$),
 the functions $\omega(x,h)$ can be estimated uniformly as the base point $x $ lies in  $M$.  Precisely, the constant $C$
 depends  only on   $C_1$ in \eqref{global}.
Finally, if the functions $c_{i,j}^k$ in \eqref{integrabile} are constant (as it happens under the assumptions of Palais' Theorem), then
 $\omega_k^j$ depend on $h$ only (not on  $x$) and the expansion
\eqref{quattordici} becomes completely independent of the base point $x$.

\step{Step 3.} Next we show  inclusion \eqref{expo}, using
\eqref{quattordici}. We claim that there is $C_0>1$ such that for all $x\in M$, the map   $E_x:B_\Eucl(0, C_0^{-1})\to (\O_x,g)$ satisfies the hypotheses of Theorem \ref{susso} (the metric $g$ is defined in \eqref{dusss}). Let $h\in B_\Eucl(0,1)\subset\R^q$.
Choose  $I\in\I(p,q)$ which satisfies the maximality condition \eqref{massimetto} at the point
 $E_x(h)$. Therefore for each $j\in\{1,\dots,q\}$, we may write
\[
 Y_j(E_x(h))= \sum_{\a=1}^p b_j^\a Y_{i_\a}(E_x(h)),
 \]
 where  by estimate \eqref{massimino} we have
  $\abs{b_j^\a}\leq 1$ for all $\alpha$.
  Inserting into \eqref{quattordici}, we discover that the map $E_x=:E$ enjoys the following expansion for all $k\in\{1,
  \dots,q\}$:
  \[
   E_{*}(\p_{h_k})=:Y_k(E(h))+\sum_{\a=1}^p \chi_k^\a Y_{i_\a}(E(h))\quad\text{for all}\quad h\in B_\Eucl(0,1)\subset\R^q,
 \]
 where the functions $\chi_k^\a$ satisfy $\abs{\chi_k^\a(x,h)}\le C\norm{h}$, for some universal constant $C$ depending on $C_1$ in \eqref{global} but not on $x\in
 M$.
 Specializing to $k=i_\beta$ with $\beta=1,\dots, q$, we get
\[
 E_*\Bigl(\frac{\p}{\p h_{i_\b}}\Bigr)=Y_{i_\b}(E(h))+
 \sum_{\a=1}^p \chi_{i_\b}^\a Y_{i_\a}(E(h))
 =\sum_{\a=1}^p (\delta_\b^\a+\chi_{i_\b}^\a) Y_{i_\a}(E(h)).
\]
In view of our estimate on the coefficients of the matrix $\chi=(\chi_{i_\b}^\a)_{\a,\b=1,\dots, p }$, we can choose $r_0$ sufficiently small (independent of $x\in M$)
so that that if $\abs{h}\le r_0$, then the operator norm of~$\chi$ satisfies say
$\abs{\chi}\le \frac{1}{2}$. Thus a Neumann series argument
ensures that $\abs{(I_p+\chi)^{-1}}\le 2$ and ultimately we have shown that for all $h\in B_\Eucl(0, r_0)\subset\R^q$ and $\ell\in\{1,\dots,q\}$, the system
\[
 \sum_{j=1}^q E_*(\p_{h_j})\xi_j=Y_\ell(E(h))
\]
has a solution $\xi\in\R^q$ with  $\abs{\xi}\le C_0$, where $C_0$ is an absolute constant. This shows that the assumptions of Theorem \ref{susso} are satisfied.

\begin{proof}[Proof of Corollary \ref{pallina}]Let
 $\gamma $ be a solution of \eqref{odessa}. Assume without loss of generality that $\norm{b}_{L^\infty}\le 1$. Then  we have
\[
x_1:=\gamma(\d/2)\in B_\rho(x,\d )\subset E_x( B_\Eucl(0,\e)),\]
by Theorem \ref{mainrough}. Note that each map $E_x$ is $C^1$ on the whole $\R^q$.
Thus,  $x_1$ belongs to the bounded set $ E_x(\{\norm{h}<\e\})$. Next, let
$x_2: =\gamma(\d )\in B_\r(x_1,\delta)$. Thus  $x_2$ belongs to the bounded set
$ E_{x_1}(B_\Eucl(0,\e))$. By iterating this argument for a sufficient number of times, say $\nu$ such that $\frac{\d}{2}\nu\ge T$,
 we discover that   the set $\gamma([0,T])$ is contained in a bounded set. Thus the solution $\gamma$ can not blow up.
\end{proof}

%
%
%
%
%

    \subsection{Examples }
In the first example we exhibit a pair of complete vector fields whose sum is not complete.

\begin{example}
Let
$
X=xy\p_x$ and $Y=xy\p_y$ in $\R^2$. Since
\[
\exp(t(X+Y))(1,1)= \Bigl(\frac{1}{1-t},\frac{1}{1-t}
\Bigr),
\]
 then $X+Y$ is not complete.
\end{example}


 In the following two examples, we exhibit finite dimensional Lie algebras of vector fields arising in complex analysis.
\begin{example}
 [CR vector fields on degenerate Siegel domains] Let $(x,t)=(x_1,x_2,t)$ be coordinates in $\R^2\times\R
 =\R^3$. Let  $p(x):=\abs{x}^4$ for $x=(x_1,x_2)\in\R^2$. Consider the pair of vector fields in~$\R^3$
  \[\begin{aligned}
 X_1&: =\p_{x_1}+\p_{2}p \,\p_t=\p_{x_1}+4x_2\abs{x}^2\p_t
 \\X_2&:=\p_{x_2}-\p_{1} p\, \p_t=\p_{x_2}-4x_1\abs{x}^2\p_t,
  \end{aligned}
\]
where $\p_j:=\frac{\p}{\p x_j}$.
Here  and in the sequel we use the notation $X_{jk}=[X_j,X_k]$, $X_{ijk}=[X_i,[X_j,X_k]]$ and so on.
We have
$
 \Lie\{X_1,X_2\}=\Span\{X_1,X_2,X_{12},X_{112}, X_{212},X_{1112}\}.
$
\end{example}

\begin{example}
 [CR vector fields on the sphere]
 We identify $\R^4$ with $\C^2$ via the identification $\C^2\ni(z_1,z_2)\simeq (x_1,x_2,y_1,y_2)\in \R^2$. Let
\[
\begin{aligned}
 X_1&:=x_1\p_{x_2} - y_1\p_{y_2}-x_2\p_{x_1}+y_2\p_{y_1}
 \\ X_2& :    = y_1\p_{x_2} +x_1\p_{y_2}- y_2\p_{x_1}-x_2\p_{y_1}.
\end{aligned}
\]
Then $X_{12}=2 (y_2\p_{x_2} -x_2\p_{y_2} +y_1\p_{x_1}-x_1\p_{y_1}  )$.
 We have
 $
  X_{112}= -4X_2\quad\text{and}\quad X_{212}=4X_1
 $
so that $\Lie\{X_1,X_2\}=\Span\{X_1,X_2,X_{12}\}$ has dimension three.
 \end{example}

Finally we exhibit an example where the Lie algebra does not have finite dimension, but our assumption \eqref{involtino} holds,
with globally bounded coefficients.
    \begin{example}\label{cierre}
 Let
again $(x,t)=(x_1,x_2,t)\in\R^3$, let $p(x_1, x_2,t)= (1+\abs{x}^2)^{1/2}-1$ and take
\[
\begin{aligned}
 X_1 & =\p_1+(\p_2p)\p_t = \p_1+(1+\abs{x}^2)^{-1/2}x_2\p_t
 \\
 X_2& =\p_2-(\p_1p)\p_t=\p_2- (1+\abs{x}^2)^{-1/2}x_1\p_t.\end{aligned}
\]
Then
\[
 [X_1,X_2]=-\frac{2+\abs{x}^2}{(1+\abs{x}^2)^{3/2}}\p_t.
\]
The H\"ormander's rank  condition is fulfilled and moreover
\[
 X_{112}=-\p_1 \frac{2+\abs{x}^2}{(1+\abs{x}^2)^{3/2}}
 \p_t=\frac{(4+\abs{x}^2)x_1}{(1+\abs{x}^2)^{5/2}}\p_t.
\]
We can write $X_{112}=:c(x)X_{12}$ with $c\in L^{\infty}(\R^2).$
The same holds for $X_{212}$.

Note that $\Lie\{X_1,X_2\}$ is not finite dimensional.  This can be seen by looking at the behavior at infinity of the coefficients of higher order commutators.

\end{example}

\bigskip 

\noindent\textbf{Acknowledgements.} The authors wish to thank the referee of the paper whose remarks 
 have led to an improvement of the manuscript.

\def\cprime{$'$} \def\cprime{$'$}
\providecommand{\bysame}{\leavevmode\hbox to3em{\hrulefill}\thinspace}
\providecommand{\MR}{\relax\ifhmode\unskip\space\fi MR }
\providecommand{\MRhref}[2]{%
  \href{http://www.ams.org/mathscinet-getitem?mr=#1}{#2}
}
\providecommand{\href}[2]{#2}

%
%
%

\end{document}